\documentclass[12pt]{amsart}
\usepackage{hyperref}
\usepackage{xcolor}
\usepackage[shortlabels]{enumitem}
\usepackage{amsmath,amssymb,amsthm,amscd}
\newtheorem{theorem}{Theorem}
\newtheorem{lemma}[theorem]{Lemma}

\newtheorem*{theorem*}{Theorem}

\theoremstyle{definition}

\usepackage[margin=3cm]{geometry}
\newcommand{\Z}{\mathbb{Z}}

\newcommand{\F}{\mathbb{F}}
\newcommand{\Q}{\mathbb{Q}}

\newcommand{\Gal}{\textrm{Gal}}

\newcommand{\im}{\textrm{im}\thinspace}

\newcommand{\Aut}{\textrm{Aut}}

\newcommand{\oo}{\mathcal{O}}

\newcommand{\ssm}{\smallsetminus}

\newcommand{\End}{\textrm{End}}

\newcommand{\GL}{\textrm{GL}}

\newcommand{\CM}{\textrm{CM}}

\newcommand{\Stab}{\textrm{Stab}}
\newcommand{\FF}{\mathcal{F}}
\newcommand{\E}{\mathcal{E}}
\newcommand{\sss}{\textrm{ss}}

\newcommand{\Gss}{G^{\sss}}

\newcommand{\tors}{\textrm{tors}}
\usepackage{mathrsfs}
\usepackage{url}
\date{}
\author{Tyler Genao}
\thanks{Email: \texttt{tylergenao@uga.edu}. 
This material is based upon work supported by the National Science Foundation Graduate Research Fellowship under Grant No. 1842396. Partial support was also provided by the Research and Training Group grant DMS-1344994 funded by the National Science Foundation.}
\title{Typically Bounding Torsion on Elliptic Curves With Rational $j$-invariant}
\begin{document}

\begin{abstract}
A family $\FF$ of elliptic curves defined over number fields is said to be \textit{typically bounded in torsion} if the torsion subgroups $E(F)[\tors]$ of those elliptic curves $E_{/F}\in \FF$ can be made uniformly bounded after removing from $\FF$ those whose number field degrees lie in a subset of $\Z^+$ with arbitrarily small upper density. For every number field $F$, we prove unconditionally that the family $\E_F$ of elliptic curves defined over number fields and with $F$-rational $j$-invariant is typically bounded in torsion. For any integer $d\in\Z^+$, we also strengthen a result on typically bounding torsion for the family $\E_d$ of elliptic curves defined over number fields and with degree $d$ $j$-invariant.
\end{abstract}
\maketitle
\tableofcontents
\section{Introduction}
\subsection{Bounds on torsion subgroups}
Let $E_{/F}$ denote an elliptic curve defined over a number field $F$. Then the Mordell--Weil theorem states that the abelian group $E(F)$ of $F$-rational points on $E$ is finitely generated. In particular, one has a decomposition
\[
E(F)\cong \Z^{r(E,F)}\oplus E(F)[\tors]
\]
where $E(F)[\tors]$ is its \textit{torsion subgroup over $F$.} In the present paper we prove results on asymptotic upper bounds for the sizes of these torsion subgroups.

In 1977, Mazur \cite{Maz77} gave a complete classification of torsion subgroups $E(\Q)[\tors]$ of all elliptic curves $E_{/\Q}$. Following this, a combination of work by Kamienny, Kenku and Momose \cite{KM88} \cite{Kam92a} \cite{Kam92b} lead to a complete classification of elliptic curve torsion subgroups $E(F)[\tors]$ over all quadratic number fields $F$. One consequence of their work is that there are finitely many possibilities for torsion subgroups over all quadratic number fields, suggesting a ``uniform bound" on torsion subgroups when the degree of the number field is fixed.

In 1996, Merel proved the ``strong uniform boundedness conjecture" suggested previously by work of Kamienny, Kenku and Momose.

\begin{theorem*}[\cite{Mer96}]
For each integer $d>0$ there exists a constant $B(d)\in\Z^+$ so that for all number fields $F/\Q$ of degree $d$ and all elliptic curves $E_{/F}$ one has
\[
\#E(F)[\emph{tors}]\leq B(d).
\]
\end{theorem*}

In lieu of Merel's result, we can ask for the sharpest upper bound $B(d)$, i.e., the largest size of any elliptic curve torsion subgroup over all degree $d$ number fields. Let us define a function which records such values: for $d\in \Z^+$ we set
\[
T(d):=\max_{E_{/F}, [F:\Q]=d}\#E(F)[\tors].
\]
Only a few values of $T(d)$ are currently known. Via a complete classification of torsion subgroups over degree $d\leq 3$ number fields, one has $T(1)=16$ \cite{Maz77}, $T(2)=24$ \cite{KM88} \cite{Kam92a} \cite{Kam92b} and $T(3)=28$ \cite{DEvHMZB}. Lower bounds for $T(d)$ are known for $4\leq d\leq 6$ \cite{JKP06} \cite{DS17}. We also have explicit upper bounds on $T(d)$ which are, unfortunately, greater than exponential in $d$ \cite{Par99}. 

Let us also define the CM-analogue of $T(d)$,
\[
T_\CM(d):=\max_{\CM ~E_{/F}, [F:\Q]=d}\#E(F)[\tors].
\]
This arithmetic function records the largest size over all CM elliptic curve torsion subgroups over all degree $d$ number fields.
One clearly has $T_\CM(d)\leq T(d)<\infty$. 

Via a complete classification of CM torsion subgroups over number fields of a fixed degree, many values of $T_\CM(d)$ are known. For example, if $d\leq 13$ then $T_\CM(d)$ is recorded \cite{CCRS14}. For all primes $\ell>5$ one has $T_\CM(\ell)=6$ \cite{BCS17}. For each $n\in\Z^+$ one has for sufficiently large primes $\ell\gg_n 0$ that $T_\CM(\ell^n)=6$ \cite{BCP17}.  In fact, for any odd $d\in\Z^+$ $T_\CM(d)$ is also known \cite{BP17}.

Bourdon, Clark and Pollack \cite{BCP17} also treat $T_\CM(d)$ as an arithmetic function and study its behavior as $d\rightarrow \infty$. As noted above, $T_\CM(d)=6$ for infinitely many $d\in\Z^+$ \cite{BCS17}. In fact, they show that $T_\CM(d)\geq 6$ for all $d\in\Z^+$, whence the ``lower order" of $T_\CM(d)$ is 
\[
\liminf_{d\rightarrow \infty} T_\CM(d)=6.
\]
Similarly, Clark and Pollack have shown the ``upper order" result (Theorem 1.1 \cite{CP17})
\[
\limsup_{d\rightarrow\infty}\frac{T_\CM(d)}{d\log\log d}=\frac{e^\gamma\pi}{\sqrt{3}}.
\]

Bourdon, Clark and Pollack \cite{BCP17} also study the \textit{typical order} of $T_\CM(d)$, e.g., its statistical behavior away from certain sets of arbitrarily small upper density.
Recall that the \textit{upper (asymptotic) density} of a subset $S\subseteq \Z^+$ is
\[
\bar{\delta}(S):= \limsup_{x\rightarrow \infty} \frac{\#(S\cap[1,x])}{x}.
\]
They prove the following typical order result on $T_\CM(d)$.
\begin{theorem*}[Theorem 1.1.(i) \cite{BCP17}]
For all $\epsilon>0$ there exists $B_\epsilon\in\Z^+$ so that
\[
\bar{\delta}(\lbrace d\in\Z^+:T_\emph{CM}(d)\geq B_\epsilon\rbrace)\leq \epsilon.
\]
\end{theorem*}
The above theorem says that for any $\epsilon>0$ we choose, by removing some subset $S\subseteq \Z^+$ of upper density $\leq \epsilon$, there exists a uniform bound $B_\epsilon$ on all CM torsion subgroups over all degree $d\not\in S$ number fields; writing this out,
\[
\#E(F)[\tors]\leq B_\epsilon
\]
for any CM elliptic curve $E_{/F}$ where $F$ is a number field whose degree $[F:\Q]\not\in S$. In light of this, we say that the family of CM elliptic curves over number fields is \textit{typically bounded in torsion.}

Compare this result to the following, which shows that torsion is decidedly not typically bounded on the family of all elliptic curves over all number fields, which we denote by $\mathcal{A}(1)$.
\begin{theorem*}[Theorem 1.7 \cite{CMP}]
For all $B\in \Z^+$ the set
\[
\lbrace d\in \Z^+:T(d)\geq B\rbrace
\]
is cofinite in $\Z^+$, whence its upper density is equal to 1.
\end{theorem*}
\subsection{Typically bounding torsion for other families}
Despite the failure of $\mathcal{A}(1)$ to be typically bounded in torsion, we saw that the proper subfamily $\mathcal{A}_\CM(1)$ of CM elliptic curves over number fields is. We will see there are also other familiar subfamilies of $\mathcal{A}(1)$ which are typically bounded in torsion.
 
Let $\FF=\lbrace {E_i}_{/F_i}\rbrace_{i\in I}$ be a family of elliptic curves $E_i$ defined over number fields $F_i$. From here on out, we assume that any such $\FF$ is ``closed under rational isomorphism", e.g., if ${E_1}_{/F}\in \FF$ and ${E_2}_{/F}$ is an elliptic curve with $E_2\cong _{F} E_1$, then ${E_2}_{/F}\in \FF$. We say that $\FF$ is \textit{typically bounded in torsion} if the torsion subgroups $E(F)[\tors]$ of elliptic curves $E_{/F}\in\FF$ can be made uniformly bounded after removing those elements whose number field degrees all lie in a certain subset of $\Z^+$ of arbitrarily small upper density. Formally stated, $\FF$ is typically bounded in torsion if for all $\epsilon>0$ there is a constant $B_\epsilon>0$ so that the set
\[
\lbrace d\in \Z^+: \exists E_{/F}\in \FF~\textrm{so that }[F:\Q]=d~\textrm{and } \#E(F)[\tors]\geq B_\epsilon\rbrace
\]
has upper density at most $\epsilon$. 

For a family $\FF$ of elliptic curves, one can define an $\FF$-analogue of $T(d)$,
\[
T_{\FF}(d):=\max_{E_{/F}\in \FF, [F:\Q]=d}\#E(F)[\tors].
\]
By Merel \cite{Mer96} one has $T_{\FF}(d)<\infty$. We see that $T_{\CM}(d)=T_{\mathcal{A}_\CM(1)}(d)$. It is clear that $\FF$ is typically bounded in torsion iff for all $\epsilon>0$ there exists a constant $B_\epsilon>0$ so that
\[
\overline{\delta}(\lbrace d\in\Z^+:T_{\FF}(d)\geq B_\epsilon\rbrace)\leq \epsilon.
\]
Compare this to Theorem 1.1.(i) \cite{BCP17} above.

One important family of elliptic curves is the ``base extension" family of all elliptic curves $E_{/F}$ which are defined over $\Q$. For example, Lozano-Robledo \cite{LR13} gives an explicit bound on all primes which divide $E(F)[\tors]$, and this bound is linear in the degree $[F:\Q]$. On the other hand, Gonz\'alez-Jim\'enez and Najman \cite{GJN20} show that $E(F)[\tors]=E(\Q)[\tors]$ when $2,3, 5, 7\nmid [F:\Q]$.
A (subtly) larger family is that of elliptic curves $E_{/F}$ with $\Q$-rational $j$-invariant. For such a family, Propp \cite{Pro18} provides (conditional) bounds on prime divisors of $E(F)[\tors]$ when $[F:\Q]$ is fixed. Clark and Pollack \cite{CP18} also show that for each $\epsilon>0$ there is a constant $C(\epsilon)>0$ for which one has polynomial bounds $\#E(F)[\tors]\leq C(\epsilon)[F:\Q]^{\frac{5}{2}+\epsilon}$.

Our primary family of study is a natural generalization of the above. For a fixed number field $F$, we consider the family of elliptic curves with $F$-rational $j$-invariant that are defined over number fields,
\[
\E_F:=\lbrace E_{/L}:[L:\Q]<\infty, j(E)\in F\rbrace.
\]
Clark, Milosevic and Pollack study the subfamily $\E'_F$ of elliptic curves $E_{/L}\in \E_F$ with $L\supseteq F$ \cite{CMP}.
They show (Theorem 1.8 \cite{CMP}) that torsion is typically bounded in $\E'_F$, provided $F$ contains no Hilbert class fields of imaginary quadratic fields and that the Generalized Riemann Hypothesis (GRH) is true. By a brief twisting argument, their proof also shows that $\E_F$ is typically bounded in torsion when one imposes these conditions.

Our main result is a generalization of Theorem 1.8 \cite{CMP}: we prove it for $\E_F$ unconditionally.
\begin{theorem}\label{mainThm}
For any number field $F$, torsion is typically bounded on the family $\E_{F}$ of elliptic curves $E$ defined over number fields and whose $j$-invariant $j(E)\in F$.
\end{theorem}
\subsection{Strategy of the proof}

Our proof of Theorem \ref{mainThm} will be broken down into several steps, following \cite{CMP}. Let us introduce a
condition P on a family $\FF$ of elliptic curves (this is condition P2 in \cite{CMP}).
\begin{enumerate}[start=1,label={\bfseries P:}]
\item There exists a constant $c:=c(\FF)\in \Z^+$ such that for all primes $\ell\in \Z^+$ and all $E_{/F}\in \FF$, if $E(F)$ has a point of order $\ell$ then
\[
\ell-1\mid c\cdot [F:\Q].
\]
\end{enumerate} 

For an integer $d\in \Z^+$ let us consider the family $\E_d$ of elliptic curves over number fields with degree $d$ $j$-invariant,
\[
\E_d:=\lbrace E_{/F}:[F:\Q]<\infty, [\Q(j(E)):\Q]=d\rbrace.
\]
By Theorems 3.2.b and 3.5 \cite{CMP} any subfamily of a finite union of $\E_d$ will be typically bounded in torsion if and only if the subfamily satisfies condition P. In particular, from the containment
\[
\E_F\subseteq \bigcup_{d=1}^n \E_d,
\]
to prove Theorem \ref{mainThm} it suffices to show that $\E_F$ satisfies condition P.

Following the proof of Theorem 4.3 \cite{CMP}, to show that condition P holds for $\E_F$ it suffices to construct a constant $c:=c(F)\in \Z^+$ such that for any prime $\ell\gg_{F} 0$ and any elliptic curve $E_{/F}$ with an $F$-rational $\ell$-isogeny,
one has for all nontrivial $R\in E[\ell]$ that
\begin{equation}\label{midReq}
\ell-1\mid c\cdot [F(R):\Q].
\end{equation}
Since $E$ has an $F$-rational $\ell$-isogeny, the image $G$ of its mod$-\ell$ Galois representation is conjugate to a subgroup of upper triangular matrices in $\GL_2(\ell)$.
We will use this fact to compare the orbits of this representation acting on the standard unit vectors, to the orbits of its semisimplification.

First, a careful analysis of the description of our isogeny character from case 1 of Theorem \ref{CMisoChar} (see the following section) will allow us to relate the semisimplification of $G$ to the mod$-\ell$ Galois representation of some CM elliptic curve $E'_{/F}$ with $F$-rational CM. By an orbit-divisibility argument, we will show that \eqref{midReq} holds when we replace $E$ with $E'$, and conclude that such a divisibility holds for $E$ with constant $c:=864$. Following this, we will show that the same constant also works when in case 2 of Theorem \ref{CMisoChar}. 

It is worth noting that Lozano-Robledo \cite{LR18} uses a similar analysis of Theorem \ref{CMisoChar} to deduce, under certain hypotheses, degree-linear bounds on the order of any rational prime power torsion point. In fact, Theorem 1.9 \cite{LR18} proves \eqref{midReq} in the form of an inequality. However, divisibility allows us to apply Theorem 3.2 \cite{CMP} and conclude that torsion is typically bounded on $\E_F$.
\subsection{An additional result}
In our proof of Theorem \ref{mainThm}, the constant $c:=864$ for which \eqref{midReq} holds is independent of the choice of base field $F$. We will use this to improve Theorem 1.10 \cite{CMP}. 

We denote by SI$(d)$ the \textit{strong isogeny conjecture in degree $d$}, which asserts that there exists a prime $\ell_0:=\ell_0(d)\in\Z^+$ so that for all primes $\ell>\ell_0$ there do not exist non-CM elliptic curves $E_{/F}$ with an $F$-rational $\ell$-isogeny when $[F:\Q]=d$. The strong isogeny conjecture SI$(d)$ would follow from a stronger conjecture of Serre which has been used to prove other results, see for example Theorem 1.6 \cite{BELOV19}.

Let us denote by LV($d$) the hypothesis that the set $S_F$ of primes from Theorem \ref{CMisoChar} (see the next section)
 depends only on $d$ for all number fields $F$ of degree $d$. One may check that LV($d$) is a weaker hypothesis than SI($d$).
\begin{theorem}\label{Ed}
Fix an integer $d\in \Z^+$. If \emph{LV}$(d)$ holds, then the family
\[
\E_{d}:=\lbrace E_{/F}: [F:\Q]<\infty, [\Q(j(E)):\Q]=d\rbrace
\]
is typically bounded in torsion.
\end{theorem}

The proof of Theorem \ref{Ed} is similar to the proof of Theorem 1.8 \cite{CMP}, which boils down to the proof of Theorem 4.3 \cite{CMP}. By Theorem 3.5 \cite{CMP}, to show that $\E_d$ is typically bounded in torsion it suffices to show it satisfies condition P. In turn, this reduces to showing that for any prime $\ell\gg_{d}0$ one has a constant $c:=c(d)\in \Z^+$ such that for any non-CM elliptic curve ${E}_{/\Q(j(E))}$ with a $\Q(j(E))$-rational $\ell$-isogeny, the divisibility
\[
\ell-1\mid c\cdot [\Q(j(E))(R):\Q]
\]
holds for all nontrivial $R\in E[\ell]$. In \cite{CMP}, the authors cite SI($d$) to exclude this case. But if one assumes LV$(d)$ instead, then our proof of Theorem \ref{mainThm} applies here and we may take $c:=864$. 
\subsection{Acknowledgments}
The author thanks Pete L. Clark for several insightful conversations about typically bounding torsion, and a careful review of earlier drafts of this paper. The author also thanks Nicholas Triantafillou for his help with the proof of Lemma \ref{GssInG}. Finally, the author thanks the referee for helpful comments and suggestions.
\subsection{Definitions and notation}
Given an extension $F/\Q$, we let $G_F:=\Gal(\overline{F}/F)$ denote the absolute Galois group of $F$. For an integer $N\in\Z^+$, let us write the mod$-N$ cyclotomic character of $F$ as
\[
\chi_N\colon G_F\rightarrow (\Z/N\Z)^\times.
\]
When $N=\ell$ is prime, we will write $\F_\ell:=\Z/\ell\Z$.

For an elliptic curve $E_{/F}$, let us write its
mod$-N$ Galois representation as
\[
\rho_{E,N}\colon G_F\rightarrow \Aut(E[N]).
\]
If $\lbrace P,Q\rbrace$ is a $\Z/N\Z$-basis for $E[N]$, then when working with explicit matrix representations we may also write
\[
\rho_{E,N,P,Q}\colon G_F\rightarrow \GL_2(N),
\]
where $\GL_2(N):=\GL_2(\Z/N\Z)$ is the general linear group of degree 2 over $\Z/N\Z$.

Suppose $E_{/F}$ has an $F$-rational cyclic $N$-isogeny $C_N$, i.e., a one-dimensional $G_F$-invariant submodule of $E[N]$. Then one has an \textit{isogeny character} which describes the action of $G_F$ on $C_N$,
\[
r\colon G_F\rightarrow (\Z/N\Z)^\times.
\]
Given a group homomorphism $f\colon G\rightarrow H$, we sometimes write $\im f$ for its image $f(G)$.
\section{Isogeny Characters}\label{Steps}
A crucial component of the proof of Theorem \ref{mainThm} is an analysis of the description of our isogeny character from a result of Larson and Vaintrob \cite{LV14}.
\begin{theorem}[Theorem 1 \cite{LV14}]\label{CMisoChar}
Let $F$ be a number field. Then there is a finite set of primes $S_F\subseteq \Z^+$ such that for all primes $\ell\not\in S_F$, if $E_{/F}$ is an elliptic curve with an $F$-rational $\ell$-isogeny, then for the corresponding isogeny character $r\colon G_F\rightarrow \F_\ell^\times$ one of the following holds:

1. There exists a CM elliptic curve $E'$ defined over $F$ such that its CM field $K:=\emph{End}(E')\otimes \Q $ is contained in $F$. Furthermore, there exists a character $\psi'\colon G_F\rightarrow \overline{\F_\ell}^\times$ for which we have both similarity
\[
(\rho_{E',\ell}\otimes_{\F_\ell}\overline{\F_\ell})(G_F)\sim \emph{im} \begin{bmatrix}
\psi'&*\\
0&\chi_\ell \psi'^{-1}
\end{bmatrix}
\]
and equality
\[
r^{12}=\psi'^{12}.
\]

2. GRH fails for $F(\sqrt{-\ell})$, and we have
\[
r^{12}=\chi_\ell^6.
\]
\end{theorem}
As per the theory of complex multiplication, the ring class field of an order $\oo$ in an imaginary quadratic field $K$ will equal $K(j(E'))$ for any $\oo$-CM elliptic curve $E'$, and all such fields will contain the Hilbert class field of $K$. In particular, for $F$ not to contain Hilbert class fields of imaginary quadratic fields is equivalent to there being no CM elliptic curve $E'$ defined over $F$ for which $\End(E')\otimes\Q\subseteq F$. For such a field $F$, if one also assumes that GRH is true then Theorem \ref{CMisoChar} tells us that for $\ell\gg_F 0$ there are no $F$-rational $\ell$-isogenies on any elliptic curve.
It is this observation, coupled with the fact that if $\FF$ satisfies condition P then one can exclude finitely many primes $\ell\in \Z^+$ and still have condition P hold for $\FF$, which shows us that $\E_F$ satisfies condition P -- this is outlined in Theorem 4.3.b \cite{CMP}. In particular, torsion is typically bounded on $\E_F$, conditionally.

To prove Theorem \ref{mainThm} we will not assume that GRH is true nor that $F$ does not contain the Hilbert class field of any imaginary quadratic field. Instead, we will leverage the extra information from Theorem \ref{CMisoChar} on our possible isogeny characters to construct a constant $c$ for which \eqref{midReq} holds for all $\ell\gg_ F 0$ and all $E_{/F}$ with an $F$-rational $\ell$-isogeny.
\section{Orbits Under the Galois Representation}
In our study of the mod$-\ell$ Galois representations of an elliptic curve $E_{/F}$, it will help to understand the orbits of the $\F_\ell[G_F]$-module $E[\ell]$ under various subgroups of $\GL_2(\ell)$.
\subsection{The semisimplification of an upper triangular subgroup of $\GL_2(\ell)$}\label{semisimpleOrbits}
Fix a prime $\ell$, and suppose $G\subseteq \GL_2(\ell)$ is a subgroup of upper triangular matrices. Let us consider its \textit{semisimplification}
\[
\Gss:=\left\lbrace \begin{bmatrix}
a&0\\
0&d
\end{bmatrix}\in \GL_2(\ell):\exists b\in \F_\ell~\begin{bmatrix}
a&b\\
0&d
\end{bmatrix}\in G\right\rbrace.
\]
The group $\GL_2(\ell)$ acts on the $\F_\ell$-vector space $V$ spanned by the column vectors $e_1:=\begin{bmatrix}
1\\0
\end{bmatrix}$ and $e_2:=\begin{bmatrix}
0\\1
\end{bmatrix}$ via multiplication on the left. Consequently, its subgroups $G$ and $\Gss$ also act on $V$ via left multiplication.

Ultimately, we will compare divisibility of the sizes of such orbits from the action of $\Gal(F(E[\ell])/F)$ on $E[\ell]^\bullet:=E[\ell]\ssm\lbrace O\rbrace$. 
The following lemma is an important step towards such a comparison.
\begin{lemma}\label{GssInG}
Let $\ell\in \Z^+$ be prime, and let $G$ be an upper triangular subgroup of $\emph{GL}_2(\ell)$. Then $G^{\emph{ss}}\subseteq G$ or $G$ is diagonalizable.
\end{lemma}
\begin{proof}
We will break the proof into three cases. In the first two cases, we will show containment $\Gss\subseteq G$.

\textbf{Case 1:} $G$ contains a non-diagonal matrix with repeated eigenvalues. Such a matrix is of the form $\begin{bmatrix}
a&b\\
0&a
\end{bmatrix}$ with $b\neq 0$. For such a matrix $\gamma$, its power $\gamma^{\ell-1}$ equals
\begin{equation}\label{1Lambda01}
\begin{bmatrix}
1&\lambda\\
0&1
\end{bmatrix}\in G
\end{equation}
where $\lambda:= (\ell-1)ba^{\ell-2}\neq 0$. It follows that $G$ contains the transvection
\[
\begin{bmatrix}
1&1\\
0&1
\end{bmatrix}=\gamma^{A}\in G,
\]
where $A\in\Z$ is taken so that $A\equiv \lambda^{-1}\pmod \ell$.
We conclude that for each $n\in \Z$ we have
\[
\begin{bmatrix}
1&n\\
0&1
\end{bmatrix}
=
\begin{bmatrix}
1&1\\
0&1
\end{bmatrix}^n
\in G.
\]
It follows that $G^{\sss}\subseteq G$. Indeed, for any element $\delta:= \begin{bmatrix}
a&b\\
0&d
\end{bmatrix}\in G$ we check that
\[
\delta\begin{bmatrix}
1&n\\
0&1
\end{bmatrix}
=
\begin{bmatrix}
a&an+b\\
0&d
\end{bmatrix},
\]
so taking $n\in \Z$ such that $n\equiv -ba^{-1}\pmod \ell$ shows that $\begin{bmatrix}
a&0\\
0&d
\end{bmatrix}\in G$.

\textbf{Case 2:} $G$ is not commutative.  Then choosing any $\gamma_1,\gamma_2\in G$ with nontrivial commutator, their commutator $\gamma_1\gamma_2\gamma_1^{-1}\gamma_2^{-1}$ must be of the form $\begin{bmatrix}
1&\lambda\\
0&1
\end{bmatrix}$ for some $\lambda\in \F_\ell^\times$. Then the steps following \eqref{1Lambda01} show that $G^{\sss}\subseteq G$.

\textbf{Case 3:} Neither case 1 nor case 2 holds. Then $G$ is commutative, and its non-diagonal matrices have distinct $\F_\ell$-rational eigenvalues. In particular, each non-diagonal matrix is diagonalizable over $\F_\ell$, and since all matrices in $G$ commute we deduce that there is a simultaneous diagonalization for these matrices over $\F_\ell$.\footnote{For a proof of this fact, see e.g. \cite{Con}.} The conclusion is that $G$ is conjugate to a subgroup of diagonal matrices.
\end{proof}

\subsection{Orbits in the diagonal case}\label{OrbitsSemi}
In our describing the orbits of the action of $\Gal(F(E[\ell])/F)$ on $E[\ell]^\bullet$, we will relate them to the orbits of $E[\ell]^\bullet$ under its semisimplification, the latter of which is contained in the subgroup of diagonal matrices. 

Let $G'\subseteq \GL_2(\ell)$ be a subgroup of diagonal matrices. Associated to this group are the two characters which describe the diagonal entries of matrices in $G'$,
\[
\chi_1\colon G'\rightarrow \F_\ell^\times,~\begin{bmatrix}
a&0\\
0&d
\end{bmatrix}\mapsto a
\]
and
\[
\chi_2\colon G'\rightarrow \F_\ell^\times,~\begin{bmatrix}
a&0\\
0&d
\end{bmatrix}\mapsto d.
\]
We will study the orbits of elements $v$ in $V^\bullet:=\F_\ell\langle e_1,e_2\rangle\ssm \lbrace 0\rbrace$ under $G'$; we will use $\oo_{G'}(v)$ to denote $G'-$orbits.

The following lemma describes the orbits under a diagonal subgroup.
\begin{lemma}\label{LemmaDiagonalOrbits}
Let $G'\subseteq \GL_2(\ell)$ be a subgroup of diagonal matrices, with diagonal characters $\chi_1$ and $\chi_2$. Then the number of orbits of the form $\oo_{G'}(ae_1)$ (resp. $\oo_{G'}(de_2)$) with $a\in \F_\ell^\times$ (resp. $d\in \F_\ell^\times$) is equal to the index $[\F_\ell^\times: \emph{im} \chi_1]$ (resp. $[\F_\ell^\times:\emph{im}\chi_2]$), and each such orbit is of size $\#\emph{im}\chi_1$ (resp. $\#\emph{im}\chi_2$). 
Orbits of the form $\oo_{G'}(ae_1+de_2)$ with $a,d\in \F_\ell^\times$ are in bijection with the orbit $\oo_{G'}(e_1+e_2)$, and the number of such orbits is $(\ell-1)^2/\#\oo_{G'}(e_1+e_2)$.
\end{lemma}
\begin{proof}
First, we claim that there are $I_1:=[\F_\ell^\times: \im \chi_1]$ distinct orbits of the form $\oo_{G'}(ae_1)$ where $a\in \F_\ell^\times$. To see this, let us check that the map
\[
\F_\ell^\times/\im \chi_1\rightarrow \lbrace\textrm{Orbits }\oo_{G'}(ae_1):a\in \F_\ell^\times\rbrace, ~\overline{a}\mapsto \oo_{G'}(ae_1)
\]
is a well-defined bijection:
\begin{enumerate}[1.]
\item Well-defined: if $a=\chi_1(\gamma)b$ then $\gamma=\begin{bmatrix}
ab^{-1}&0\\
0&d
\end{bmatrix}$ for some $d\in\F_\ell^\times$, and thus
\[
\oo_{G'}(be_1)=\oo_{G'}(\gamma be_1)=\oo_{G'}(ae_1).
\]
\item Injective: suppose that $\oo_{G'}(ae_1)=\oo_{G'}(be_1)$. Then for some $\gamma\in G'$ we have $ab^{-1}e_1=\gamma e_1$, whence we have $\chi_1(\gamma)=ab^{-1}$, i.e., $a\cdot \im \chi_1\equiv b\cdot \im \chi_1$ in $\F_\ell^\times/\im \chi_1$.
\item Surjective: obvious.
\end{enumerate} 
We conclude that $e_1$ contributes $I_1$ distinct orbits of the form $\oo_{G'}(ae_1)$ with $a\in\F_\ell^\times$, each of size $\#\im \chi_1$. 
An identical argument shows that $e_2$ contributes $I_2:=[\F_\ell^\times: \im \chi_2]$ distinct orbits of the form $\oo_{G'}(de_2)$ where $d\in \F_\ell^\times$, each of size $\#\im\chi_2$.
The remaining orbits are of the form $\oo_{G'}(ae_1+de_2)$ where $ad\neq 0$. A similar argument shows that $\#\oo_{G'}(ae_1+de_2)=\#\oo_{G'}(e_1+e_2)$ when $ad\neq 0$.

Since the orbits above partition $V^\bullet:=\F_\ell\langle e_1,e_2\rangle\ssm \lbrace 0\rbrace$, one counts sizes and finds that
\[
I_1\#\im\chi_1+I_2\#\im\chi_2+X\#\oo_{G'}(e_1+e_2)=\ell^2-1
\]
where $X$ denotes the number of distinct orbits of the form $\oo_{G'}(ae_1+de_2)$ with $ad\neq 0$. In particular, since $I_1\#\im\chi_1=I_2\#\im\chi_2=\ell-1$ we conclude that
\[
X\#\oo_{G'}(e_1+e_2)=(\ell-1)^2.
\]
\end{proof}
\subsection{Uniform orbit divisibility}\label{orbDivSection}
Suppose that a group $G$ acts on a finite set $X$. For an element $x\in X$ let us use $\oo_G(x)$ to denote its orbit. 
Let us say that an integer $M\in\Z^+$ \textit{uniformly divides} $G$-orbits if there exists an integer $c\in\Z^+$ such that for all $x\in X$ one has
\[
M\mid c\cdot \#\oo_G(x).
\]
Clearly, any integer $M>0$ will uniformly divide $G$-orbits if we take $c:=M$. The relevance of uniform divisibility comes from 
the following lemma, which will allow us to prove \eqref{midReq} when we replace the image $G$ of our mod$-\ell$ Galois representation with a particular subgroup, and then ``carry over" this constant $c$ to the divisibility \eqref{midReq} for $E$. As will be shown in the proof of \eqref{midReq}, such a constant will not depend on $M:=\ell-1$, nor on $E$.
\begin{lemma}\label{LemmaUniformDivisibility}
Let $G$ be a group which acts on a finite set $X$ on the left, and let $H\subseteq G$ be a finite index subgroup. 
\begin{enumerate}[1.]
\item If for an integer $M>0$ there is a constant $c\in\Z^+$ such that for all $x\in X$ we have $M\mid c\cdot \#\oo_H(x)$, then we also have for all $x\in X$ the divisibility $M\mid c\cdot \#\oo_G(x)$.
\item If for an integer $M>0$ there is a constant $C\in\Z^+$ such that for all $x\in X$ we have $M\mid C\cdot \#\oo_G(x)$, then we also have for all $x\in X$ the divisibility $M\mid C\cdot [G:H]\cdot \#\oo_H(x)$.
\end{enumerate}
\end{lemma}
\begin{proof}
Let us write the left coset representatives of $H$ as $g_1^{-1},\ldots, g_{[G:H]}^{-1}$. Then one has
\[
\oo_G(x)=\bigsqcup _{i=1}^{[G:H]} \oo_H(g_i x).
\]
It follows that for all $x\in X$ one has
\[
\#\oo_G(x)=\sum_{i=1}^{[G:H]}\#\oo_H(g_i x),
\]
from which the first part follows. The second part follows from the fact that for each $x\in X$ one has $\#\oo_G(x)\mid [G:H]\cdot \#\oo_H(x)$.
\end{proof}
\subsection{Action on $\ell$-torsion}
Let $F$ be a number field and $E_{/F}$ an elliptic curve. Throughout this paper we are considering for various primes $\ell\in\Z^+$ the mod$-\ell$ Galois representation of the absolute Galois group $G_F$ acting on the $\F_\ell$-module $E[\ell]$. At each level $\ell$, our action induces a faithful action $\rho_{E,\ell}\colon\Gal(F(E[\ell])/F)\hookrightarrow\Aut(E[\ell])$. We will often work with a basis $\lbrace P,Q\rbrace$ of $E[\ell]$ in mind, in which case the image $G:= \rho_{E,\ell,P,Q}(\Gal(F(E[\ell])/F))$ of our Galois representation is identified with a subgroup of $\GL_2(\ell)$.

For the image $G$ of a mod$-\ell$ Galois representation as above, Our $G$-orbits of points $R\in E[\ell]^\bullet$ are 
\[
\oo_G(R)=\lbrace R^\sigma:\sigma\in \Gal(F(E[\ell])/F)\rbrace,
\]
and the stabilizer of $R$ is the subgroup of automorphisms which fix $R$,
\[
\Stab (R)=\Gal(F(E[\ell])/F(R)).
\]
By the orbit-stabilizer theorem, the size of an element's orbit is the index of its stabilizer, whence the degree of an $\ell$-torsion point $R\in E[\ell]$ satisfies
\[
[F(R):F]=\#\oo_G(R).
\]
Consequently, torsion points in the same orbit share the same degree over $F$.
\section{Part One of the Proof: Allowing Rationally Defined CM}\label{ProofWithRationalCM}
To recapitulate our goals: we will prove Theorem \ref{mainThm}, which is an unconditional form of Theorem 1.8 \cite{CMP}. This means we will allow our number field $F$ to contain Hilbert class fields of imaginary quadratic fields, and we will not assume that GRH is true. Following the proof of Theorem 4.3 \cite{CMP}, we will show the following: for all primes $\ell\gg_F 0$ and all non-CM elliptic curves $E_{/F}$ whose mod$-\ell$ Galois representation image $G:=\rho_{E,\ell}(G_F)$ is contained in a Borel subgroup -- i.e., $G$ is conjugate to a subgroup of upper triangular matrices -- one has that all $G$-orbits are uniformly divisible by $\ell-1$ with an absolute constant, i.e., there exists $c\in\Z^+$ such that for all $\ell$-torsion $R\in E[\ell]^\bullet$, one has that Equation \eqref{midReq} holds,
\[
\ell-1\mid c\cdot [F(R):\Q].
\]

For this section, let us assume that GRH is true. This will land us in case 1 of Theorem \ref{CMisoChar}, where our isogeny character agrees with a CM ``mod$-\ell$ associated character" up to twelfth powers.
\subsection{Reducing to the split prime case}\label{reduceSplit}
As per Remark 3.1 \cite{CMP} we can exclude any finite amount of primes from our consideration. We will show why this lets us assume that $\ell\gg_F 0$ is split in the orders which come from Theorem \ref{CMisoChar}. 

For any $\oo$-CM elliptic curve $E'$ defined over $F$ where $K:=\oo\otimes \Q\subseteq F$, we have that the ring class field $K(\oo)=K(j(E'))\subseteq F$. Since $F$ may contain only finitely many such ring class fields,\footnote{This is an application of Heilbronn's classic result on imaginary quadratic class numbers \cite{Hei34} towards showing that the class number of an imaginary quadratic order $\oo$ tends to infinity as its discriminant $\Delta(\oo)\rightarrow -\infty$.} we need only to consider finitely many imaginary quadratic orders $\oo$ and (up to isomorphism) finitely many $\oo$-CM $E'_{/F}$. In particular, for the rest of the section let us assume that $\ell$ is unramified in any imaginary quadratic order $\oo$ with $[K(\oo):\Q]\leq [F:\Q]$. It will follow that our CM associated character $\psi'$ in Theorem \ref{CMisoChar} will give us a diagonalization of the image of our CM Galois representation over $\overline{\F_\ell}$.

Suppose that $E_{/F}$ is an elliptic curve whose mod$-\ell$ Galois representation image is contained in a Borel subgroup, i.e., suppose that $E$ has an $F$-rational $\ell$-isogeny. Let $r\colon G_F\rightarrow \F_\ell^\times$ denote its isogeny character. Assuming $\ell\gg_F 0$ as above and $\ell\not\in S_F$ where $S_F$ is as in Theorem \ref{CMisoChar}, there exists an elliptic curve $E'_{/F}$ with CM by an imaginary quadratic order $\oo$ whose CM field $K:=\oo\otimes \Q\subseteq F$, and there exists a character $\psi'\colon G_F\rightarrow \overline{\F_\ell}^\times$ for which both
\begin{equation}\label{DiagonalOverFlBar}
(\rho_{E',\ell}\otimes \overline{\F_\ell})(G_F)\sim \im \begin{bmatrix}
\psi'&0\\
0&\chi_\ell \psi'^{-1}
\end{bmatrix}
\end{equation}
and
\begin{equation}\label{twelfthPowers}
r^{12}=\psi'^{12}.
\end{equation}
We note that Equation \eqref{twelfthPowers} implies 
\begin{equation}\label{ordersGCDCharacters}
\gcd(12, \#\psi'(G_F))\cdot \#r(G_F)=\gcd(12,\# r(G_F))\cdot \#\psi'(G_F).
\end{equation}

We may assume that $\ell$ is odd. Suppose that $\ell$ is inert in $\oo$.
In this case, up to conjugation the image $\rho_{E',\ell}(G_F)$ of the CM Galois representation lands in the non-split Cartan subgroup of $\GL_2(\ell)$,
\[
C_{ns}(\ell):=\left\lbrace \begin{bmatrix}
a&b\epsilon\\
b&a
\end{bmatrix}:(a,b)\in \F_\ell^2, (a,b)\neq (0,0)\right\rbrace;
\]
here, $\epsilon$ is the least positive generator of $\F_\ell^\times$. For $\ell>2$, one has an isomorphism $\F_\ell(\sqrt{\epsilon})^\times\cong C_{ns}(\ell)$ by sending each element $a+b\sqrt{\epsilon}\in\F_\ell(\sqrt{\epsilon})^\times$ to its regular matrix representation $\begin{bmatrix}a&b\epsilon\\
b&a\end{bmatrix}$. The eigenvalues of such matrices are $a\pm b\sqrt{\epsilon}$, and since the Galois group $\Gal(\F_\ell(\sqrt{\epsilon})/\F_\ell)$ is generated by the Frobenius automorphism $\alpha\mapsto \alpha^\ell$, these eigenvalues are Galois conjugates. This implies that $\psi'$ and $\chi_\ell\psi'^{-1}$ are Galois conjugates.
Therefore, Equation \eqref{DiagonalOverFlBar} shows that
\[
\#\rho_{E',\ell}(G_F)=\#\psi'(G_F).
\]
Since $\#r(G_F)\mid (\ell-1)$, comparing this to \eqref{ordersGCDCharacters} shows that
\begin{equation}\label{CMGalRepDivl-1}
\#\rho_{E',\ell}(G_F)\mid 12(\ell-1).
\end{equation}
On the other hand, $\oo$-CM theory with $\ell$ inert tells us that for any torsion point $P'\in E'[\ell]^\bullet$ one has
\[
(\ell^2-1)\mid \#\oo^\times\cdot [F:K(\oo)]\cdot [F(P'):F]
\]
where $\oo^\times$ denotes the unit group of $\oo$, see Theorem 7.2 \cite{BC18}.
In particular, from 
\[
[F(P'):F]\mid [F(E'[\ell]):F]=\#\rho_{E',\ell}(G_F)
\]
we find that
\[
(\ell^2-1)\mid \#\oo^\times\cdot [F:K(\oo)]\cdot \#\rho_{E',\ell}(G_F).
\]
Comparing this with \eqref{CMGalRepDivl-1}, we conclude that $\ell+1\mid 72[F:K(\oo)]$, which bounds $\ell$ that come from the case where an elliptic curve has an $F$-rational $\ell$-isogeny for $\ell$ inert in an imaginary quadratic order in $F$. 

Henceforth, we assume that $\ell$ splits in any order $\oo$ we are considering. This implies that the CM image $\rho_{E',\ell}(G_F)$ lands in a split Cartan subgroup of $\GL_2(\ell)$, and so $\rho_{E',\ell}(G_F)$ is diagonalizable over $\F_\ell$. In particular, our character $\psi'$ must be the isogeny character of an $F$-rational $\ell$-isogeny $\langle P'\rangle$ of $E'$.  
\subsection{The proof}\label{nonCommCase}
As per Lemma 4, by conjugating if necessary we may assume that with respect to a basis $\lbrace P,Q\rbrace$ of $E[\ell]$, $G$ is both upper triangular and contains its semisimplification $\Gss$. Without loss of generality, let us suppose that $\langle P\rangle$ is an $F$-rational $\ell$-isogeny whose isogeny character is $r\colon G_F\rightarrow \F_\ell^\times$. Then we have
\[
G:=\rho_{E,\ell, P,Q}(G_F)=\im \begin{bmatrix}
r&*\\
0&\chi_\ell r^{-1}
\end{bmatrix}
\]
where $*\colon G_F\rightarrow \F_\ell^\times$ is not necessarily a character. 
Let us assume that $\ell\gg_F 0$ as in the previous subsection and $\ell\not\in S_F$. Then Theorem \ref{CMisoChar} applies: we have an $\oo$-CM elliptic curve $E'_{/F}$ with an $F$-rational $\ell$-isogeny $\langle P'\rangle$, say with isogeny character $\psi'\colon G_F\rightarrow \F_\ell^\times$, and there is a complementary basis element $Q'$ to $P'$ so that the image of the mod$-\ell$ Galois representation of $E'$ is diagonal with respect to the basis $\lbrace P',Q'\rbrace$. 
Let us write
\[
G':=\rho_{E',\ell, P',Q'}(G_F)=\im \begin{bmatrix}
\psi'&0\\
0&\chi_\ell \psi'^{-1}
\end{bmatrix}.
\]
We have written our actions here as left multiplication on the two-dimensional $\F_\ell$-vector space $V:=\F_\ell\langle e_1,e_2\rangle$ with $e_i$ the unit column vectors. So for each $v=ae_1+de_2\in V$, under $G$ one has that $v$ corresponds to $R_v:=aP+dQ\in E[\ell]$, and under $G'$ one has that $v$ corresponds to $R'_v:=aP'+dQ'\in E'[\ell]$. 

Let $C_s(\ell)$ denote the split Cartan subgroup of $\GL_2(\ell)$, i.e., the subgroup of diagonal matrices. Then since both $\Gss\subseteq G$ and $r^{12}=\psi'^{12}$, we have the inclusions
\[
C_s(\ell)\supseteq G'\supseteq G'^{12}\subseteq \Gss\subseteq G.
\]
Suppose we have a constant $C\in\Z^+$ so that uniform $C_s(\ell)$-orbit divisibility by $\ell-1$ holds with $C$: i.e., for all $v\in\ V^\bullet$
\[
\ell-1\mid C\cdot \#\oo_{C_s(\ell)}(v).
\]
Then by Lemma \ref{LemmaUniformDivisibility}, the containment $C_s(\ell)\supseteq G'$ implies that for all $v\in V^\bullet$
\[
\ell-1\mid C\cdot [C_s(\ell):G']\cdot \#\oo_{G'}(v).
\]
With uniform divisibility of $G'$-orbits via the constant $C\cdot [C_s(\ell):G']$, we apply Lemma \ref{LemmaUniformDivisibility} to the containment $G'\supseteq {G'}^{12}$ and find that for all $v\in V^\bullet$ 
\[
\ell-1\mid C\cdot [C_s(\ell):G']\cdot [G':{G'}^{12}]\cdot \#\oo_{{G'}^{12}}(v).
\]
A final application of Lemma \ref{LemmaUniformDivisibility} towards the containment ${G'}^{12}\subseteq G$ implies that for all $v\in V^\bullet$ one has
\[
\ell-1\mid C\cdot [C_s(\ell):G']\cdot [G':G'^{12}]\cdot \#\oo_{G}(v),
\]
i.e., for all $R\in E[\ell]^\bullet$
\[
\ell-1\mid C\cdot [C_s(\ell):G']\cdot [G':G'^{12}]\cdot [F(R):F].
\]
Since $[G':G'^{12}]\mid (12)^2$, this simplifies to
\begin{equation}\label{divUnknowns}
\ell-1\mid 144 C\cdot [C_s(\ell):G']\cdot [F(R):F].
\end{equation}
With \eqref{divUnknowns} in mind, we will conclude that \eqref{midReq} holds with an absolute constant by showing that uniform $C_s(\ell)$-orbit divisibility by $\ell-1$ holds with an absolute constant, and that the index $[C_s(\ell):G']$ is bounded in terms of $[F:\Q]$.

First, we claim that uniform $C_s(\ell)$-orbit divisibility by $\ell-1$ holds with constant $C:=1$ -- i.e., all $C_s(\ell)$-orbit sizes are divisible by $\ell-1$. By Lemma \ref{LemmaDiagonalOrbits}, there are exactly three orbits of $C_s(\ell)$ under its action on $V^\bullet$: they are $\oo_{C_s(\ell)}(e_1), \oo_{C_s(\ell)}(e_2)$ and $\oo_{C_s(\ell)}(e_1+e_2)$, and each are of size $\ell-1, \ell-1$ and $(\ell-1)^2$, respectively. This proves our first claim.

Next, we will show that the index $[C_s(\ell):G']$ is bounded in terms of $[F:\Q]$. 
In general, for an $\oo$-CM elliptic curve $E'_{/F}$ with $F$-rational CM -- i.e., $\oo\otimes \Q\subseteq F$ -- one has for any $N\in \Z^+$ that the image of its mod-$N$ Galois representation will land inside the mod-$N$ Cartan subgroup $C_N(\oo):=(\oo/N\oo)^\times$. In fact, the indices of \textit{all} mod$-N$ Galois representation images of $E'$ are uniformly bounded in terms which involve $F$.
\begin{theorem}[Corollary 1.5 \cite{BC18}]\label{CMOIT}
Let $K$ be an imaginary quadratic field. Then the index of the image of the mod$-N$ Galois representation of any $\oo$-CM elliptic curve $E'_{/F}$ with $\oo\otimes \Q=K\subseteq F$ satisfies 
\[
[C_N(\oo): \rho_{E',N}(G_F)]\mid \#\oo^\times[F:K(j(E'))].
\]
\end{theorem}
Since $\ell$ splits in $\oo$, the mod$-\ell$ Cartan subgroup is split, i.e.,  $C_N(\oo)=C_s(\ell)$. Then by Theorem \ref{CMOIT}, we have that the index of the image of our CM mod$-\ell$ Galois representation $G'$ satisfies
\[
[C_s(\ell):G']\mid 6[F:\Q].
\]
Combining this with \eqref{divUnknowns} shows that for all $R\in E[\ell]^\bullet$ one has
\[
\ell-1\mid 864[F:\Q]\cdot [F(R):F].
\]
We conclude that \eqref{midReq} holds for $E$ with constant $c:=864$.
\section{Part Two of the Proof: Removing GRH}\label{ProofWithGRH}
Let $E_{/F}$ be an elliptic curve with an $F$-rational $\ell$-isogeny for $\ell\gg_F 0$. We will show that \eqref{midReq} holds with constant $c:=864$ without assuming that GRH is true. In doing so, there is an additional case which may appear from Theorem \ref{CMisoChar} that will be dealt with in this section.

Let $\langle P\rangle$ be an $F$-rational $\ell$-isogeny of $E$; let us write its isogeny character as $r\colon G_F\rightarrow \F_\ell^\times$. Suppose we are in case 2 of Theorem \ref{CMisoChar}, so that
\begin{equation}\label{GRHfalseCharacters}
r^{12}=\chi_\ell^6
\end{equation}
where $\chi_\ell\colon G_F\rightarrow \F_\ell^\times$ is the mod$-\ell$ cyclotomic character. 

As per Lemma \ref{GssInG}, let us choose a basis $\lbrace P,Q\rbrace$ of $E[\ell]$ for which $G:=\rho_{E,\ell, P,Q}(G_F)$ is upper triangular and $\Gss\subseteq G$. Then we have
\[
G=\im\begin{bmatrix}
r&*\\
0&\chi_\ell r^{-1}
\end{bmatrix}
\]
and thus
\[
\Gss=\im \begin{bmatrix}
r&0\\
0&\chi_\ell r^{-1}
\end{bmatrix}.
\]
By Equation \eqref{GRHfalseCharacters} we have $(\chi_\ell r^{-1})^6=r^6$, whence we deduce that 
\[
(\Gss)^6=\im\begin{bmatrix}
r^6&0\\
0&r^6
\end{bmatrix}
\]
is a subgroup of scalars.
As per the inclusions $(\Gss)^6\subseteq \Gss\subseteq G$, to show that \eqref{midReq} holds it suffices
by Lemma \ref{LemmaUniformDivisibility} to show that $\ell-1$ uniformly divides $(\Gss)^6$-orbits with constant $864[F:\Q]$. In fact, we will show that for all $v\in \F_\ell\langle e_1,e_2\rangle^\bullet$ one has
\[
\ell-1\mid 36[F:\Q]\cdot \#\oo_{(\Gss)^6}(v).
\]

By Lemma \ref{LemmaDiagonalOrbits} the sizes of orbits of the form $\oo_{(\Gss)^6}(ae_1)$ and $\oo_{(\Gss)^6}(de_2)$ with $ad\neq 0$ are $\#r^6(G_F)$, and the rest of the orbits $\oo_{(\Gss)^6}(ae_1+de_2)$ with $ad\neq 0$ share the same size as the orbit $\oo_{(\Gss)^6}(e_1+e_2)$. Since the action under $\Gss$ is scalar, it is clear that $\#\oo_{(\Gss)^6}(e_1+e_2)=\#r^6(G_F)$. Therefore, to prove \eqref{midReq} with constant $c:=864$ it suffices to show that
\begin{equation}\label{divGRHCase}
\ell-1\mid 864[F:\Q]\cdot \#r^6(G_F).
\end{equation}

Since $r(G_F)$ is cyclic, we observe that
\[
\#r^{12}(G_F)=\frac{\#r(G_F)}{\gcd(12, \#r(G_F))}.
\]
Since $\chi_\ell(G_F)$ is also cyclic, the image of its sixth power has size
\[
\#\chi_\ell^{6}(G_F)=\frac{\ell-1}{\gcd(6, \ell-1)}.
\]
Since $\chi_\ell^6=r^{12}$, we compare these sizes and find that
\[
\#\chi_\ell(G_F)\cdot \gcd(12,\#r(G_F))=\gcd(6, \#\chi_\ell(G_F))\cdot \#r(G_F),
\]
whence we have
\[
\#\chi_\ell(G_F)\mid 6\cdot \#r(G_F).
\]
Since we also have
\[
\#r(G_F)=\#r^6(G_F)\cdot \gcd(6,\#r(G_F))
\]
it follows that
\[
\#\chi_\ell(G_F)\mid 36\cdot \#r^6(G_F).
\]
Finally, since $\chi_\ell\colon G_\Q\rightarrow \F_\ell^\times$ is surjective, we have $[\F_\ell^\times:\chi_\ell(G_F)]\mid [F:\Q]$, and so we conclude that
\[
\ell-1\mid 36[F:\Q]\cdot \#r^6(G_F).
\]
This proves \eqref{divGRHCase}, which concludes our proof of Theorem \ref{mainThm}.

\end{document}